\documentclass[12pt, reqno]{amsart}
\usepackage{amssymb}
\input xypic
\xyoption{all}

\newtheorem{thm}{Theorem}

\newtheorem{cor}[thm]{Corollary}
\newtheorem{lem}[thm]{Lemma}

\theoremstyle{remark}
\newtheorem{rem}[thm]{Remark}

\newtheorem{ques}[thm]{Question}

\newcommand{\derk}{{\rm Der_k}}

\newcommand{\tor}{\xymatrix{\ar@{-->}[r]&}}

\begin{document}

\author{Rene Baltazar}

\thanks{Research of R. Baltazar  was partially supported by CAPES of Brazil}

\title[Commuting derivations and Simplicity]{Simplicity and commutative bases of derivations in polynomial and power series rings}
\maketitle

\centerline{Universidade Federal do Rio Grande do Sul, Brasil}

\centerline{rene.baltazar@ufrgs.br}

\begin{abstract} The first part of the paper will describe a recent result of K. Retert in (\cite{Ret}) for $k[x_1,\ldots,x_n]$ and $k[[x_1,\ldots,x_n]]$. This result states that if $\mathfrak{D}$ is a set of commute $k$-derivations of $k[x,y]$ such that both $\partial_x \in \mathfrak{D}$ and the ring is $\mathfrak{D}$-simple, then there is $d \in \mathfrak{D}$ such that $k[x,y]$ is $\{\partial_x ,d\}$-simple. As applications, we obtain relationships with known results of A. Nowicki on commutative bases of derivations.

\end{abstract}

\section{introduction}

Let $k$ a field of characteristic zero and $R$ denote either the ring $k[x_1,\ldots,x_n]$ of polynomials over $k$ or the ring $k[[x_1,\ldots,x_n]]$ of formal power series over $k$.

A $k$-\emph{derivation} $d:R\rightarrow R$ of $R$ is a $k$-linear map such that $d(ab)=d(a)b+ad(b)$ for any $a,b \in R$. Denoting by $\derk(R)$ the set of all $k$-derivations of $R$, let $\mathfrak{D} \subset \derk(R)$ be a nonempty family of $k$-derivations. An ideal $I$ of $R$ is called $\mathfrak{D}$-\emph{stable} if $d(I) \subset I$ for all $d \in \mathfrak{D}$. For example, the ideals $0$ and $R$ are always $\mathfrak{D}$-stable. If $R$ has no other $\mathfrak{D}$-stable ideal it is called $\mathfrak{D}$-\emph{simple}. When $\mathfrak{D}= \{d\}$, $d$ is often called a \emph{simple derivation}.

The commuting derivations has been studied by several authors:  Jiantao Li and Xiankun Du (\cite{Jidu}), S. Maubach (\cite{Mau}), A. Nowicki (\cite{Now86}), A.P. Petravchuk (\cite{Pet}), K. Retert (\cite{Ret}), A. van den Essen (\cite{Van}). For example, it is well known that each pair of commuting linear operators on a finite dimensional vector space over an algebraically field has a common eigenvector; in (\cite{Pet}), A.P. Petravchuk proved an analogous statement for derivation of $k[x,y]$ over any field $k$ of characteristic zero. More explicitly, if two derivations of $k[x,y]$ are linearly independent over $k$ and commute, then they have a common Darboux polynomial or they are Jacobian derivations; in (\cite{Jidu}), the authors proved the same result for $k[x_1,\ldots,x_n]$ and $k[[x_1,\ldots,x_n]]$. However, we observe that this result has already been proved by A. Nowicki in ({Now86}) for both rings. Another interesting result was proved by A. Nowicki in (\cite[Theorem 5.]{Now86}) which says that the famous Jacobian Conjecture in $K[x_1,\ldots,x_n]$ is equivalent to the assertion that every commutative basis of $\derk(R)$ is locally nilpotent.

Let $\mathfrak{D}$ a set of commute $k$-derivations of $k[x]$, then $k[x]$ is $\mathfrak{D}$-simple if and only if it is $d$-simple for some $k$-derivation $d \in \mathfrak{D}$ (see \cite[Corollary 2.10.]{Ret}). In $k[x,y]$, as pointed out in \cite{Ret}, up to scalar multiples, these are only sets $\mathfrak{D}$ of two commuting, nonsimple $k$-derivations such that both $\partial_x \in \mathfrak{D}$ and $R$ is $\mathfrak{D}$-simple. Motivated by this, we analyze this result in \cite{Ret} for $R$ and then we propose some connections with known results on commutative bases of derivations in $R$. More precisely, the derivations $\partial_{x_1},\ldots,\partial_{x_{n-1}}$ are not simple $k$-derivations of $R$; however, as will be shown, they can be part of a set $\mathfrak{D}$ of $n$ commuting, nonsimple $k$-derivations such that $R$ is $\mathfrak{D}$-simple. A trivial example is $\mathfrak{D}=\{\partial_{x_1},\ldots,\partial_{x_{n}}\}$. Using the notations in \cite{Now86}, we give a nontrivial commutative base containing only nonsimple $k$-derivations of the free $R$-module $\derk (R)$ such that $R$ is $\mathfrak{D}$-simple and if the Jacobian Conjecture is true in $k[x_1,\ldots,x_n]$, as a consequence of the (\cite[Theorem 5.]{Now86}), we obtain a family of locally nilpotent derivations.

\section{commuting derivations and simplicity}

\begin{lem}\label{lem1.1}
The set of all $k$-derivations of $R$ that commute with $\partial_{x_1},\ldots,\partial_{x_{n-1}}$ is
$$\{p(x_n)\partial_{x_1}+q_1(x_n)\partial_{x_2}+\ldots+q_{n-1}(x_n)\partial_{x_n}\}$$
such that $p, q_1,\ldots,q_{n-1} \in k[x_n]$ \rm(or $k[[x_n]]$ if $R=k[[x_1,\ldots,x_n]]$).
\end{lem}

\begin{proof}
Is clear that all derivations of this form commute with the derivations $\partial_{x_1},\ldots,\partial_{x_{n-1}}$. For the converse, let $$d^*=p^*(x_1,\ldots,x_n)\partial_{x_1}+q^*_1(x_1,\ldots,x_n)\partial_{x_2}+\ldots+q^*_{n-1}(x_1,\ldots,x_n)\partial_{x_n}$$ be a $k$-derivation of $R$ that commute with $\partial_{x_1},\ldots,\partial_{x_{n-1}}$. Then $$0=d^*(\partial_{x_i}(x_1))=\partial_{x_i}(d^*(x_1))=\partial_{x_i}(p^*(x_1,\ldots,x_n))$$ for all $i=1,\ldots,n-1$. Thus, $p^*(x_1,\ldots,x_n) \in K[x_n](K[[x_n]])$. Similarly, we can proof that $q^*_{i}(x_1,\ldots,x_n) \in K[x_n](K[[x_n]])$.
\end{proof}

Let $\mathfrak{D}= \{\delta_1,\ldots,\delta_s\}$ be any finite set of $k$-derivation of $R$ that commute with $\partial_{x_1},\ldots,\partial_{x_{n-1}}$, but not necessarily each with other. By Lemma \ref{lem1.1}, each $\delta_i$ is of the form $$\delta_i=p_i(x_n)\partial_{x_1}+q_{(1,i)}(x_n)\partial_{x_2}+\ldots+q_{(n-1,i)}(x_n)\partial_{x_n}.$$

We denote by $v_{n-1}(x_n)$ the greatest common divisor of $q_{(n-1,1)},\ldots,q_{(n-1,s)}$. We have the following characterization for the simplicity of $R$.

\begin{lem} \label{lem1.2}
Using the notations above, $R$ is $\mathfrak{D}$-simple if and only if $v_{n-1}(x_n)$ is a unit in $k[x_n]$ \rm(or $k[[x_n]]$).
\end{lem}

\begin{proof}
If all $q_{(n-1,i)}=0$, so all the $k$-derivation in $\mathfrak{D}$ stabilize the nonzero ideal $\langle x_n\rangle$; in this case, $R$ is not $\mathfrak{D}$-simple. Then we assume that some $q_{(n-1,i)}\neq0$. If $v_{n-1}(x_n)$ is not a unit, each $\delta_i$ stabilize the nontrivial ideal $\langle v_{n-1}(x_n)\rangle$. Therefore $R$ is not $\mathfrak{D}$-simple.

Conversely, assume that $v_{n-1}(x_n)$ is a unit and notice that, in this case, there are polynomials $r_i(x_n)$ such that $$\sum_{i=1}^s r_i(x_n)q_{(n-1,i)}(x_n)=v_{n-1}(x_n),$$ multiplicand by the inverse of $v_{n-1}(x_n)$, we may assume that $\sum_{i=1}^s r_i(x_n)q_{(n-1,i)}(x_n)=1.$ Without loss of generalization, let $\delta_1=\partial_{x_1}$,\ldots,$\delta_{n-1}=\partial_{x_{n-1}}$; thus, $$\sum_{i=n}^s r_i(x_n)q_{(n-1,i)}(x_n)=1.$$ Let $I$ be a $\mathfrak{D}$-ideal. Then since $I$ is stabilized by each $\delta_i$, $I$ is stabilized by $r_i(x_n)\delta_i$, and, then, by the $k$-derivation $$\sum_{i=n}^s r_i(x_n)\delta_i=\left(\sum_{i=n}^s r_i(x_n)p_i(x_n)\right)\partial_{x_1}+ \ldots+\left(\sum_{i=n}^s r_i(x_n)q_{(n-1,i)}(x_n)\right)\partial_{x_n}.$$ Therefore, $I$ is stabilized by $\partial_{x_1},\ldots,\partial_{x_{n-1}}$ and a $k$-derivation of the form $$u_1(x_n)\partial_{x_1}+\ldots+u_{n-1}(x_n)\partial_{x_{n-1}}+ \partial_{x_n}$$ for $u_i(x_n) \in K[x_n](K[[x_n]])$. Thus $I$ is stabilized by $\partial_{x_1},\ldots,\partial_{x_{n-1}},\partial_{x_n}$ and, then, we deduce that $I$ must be a trivial ideal.
\end{proof}

Note that until now we only assume that all the $k$-derivations commute with $\partial_{x_1},\ldots,\partial_{x_{n-1}}$, not that all elements commute with each other. Using the previous lemmas, the following theorem will show that if $R$ is $\mathfrak{D}$-simple under a set $\mathfrak{D}$ of commuting $k$-derivations that contains $\partial_{x_1},\ldots,\partial_{x_{n-1}}$, then $R$ is simple under a subset of $n$ commuting nonsimple $k$-derivations.

\begin{thm} \label{thm1}Let $\mathfrak{D}$ a set of $k$-derivations of $R$ such that $\partial_{x_1},\ldots,\partial_{x_{n-1}} \in \mathfrak{D}$. Then the derivations of $\mathfrak{D}$ commute with each other if and only if one the following two cases holds.

(a) Each element $\delta_i$ of $\mathfrak{D}$ has the form $\delta_i=h_1^i(x_n)\partial_{x_1}+\ldots+h_{n-1}^i(x_{n})\partial_{x_{n-1}}$, for some $h_j^i(x_n) \in k[x_n]$ \rm(or $k[[x_n]]$).

(b) There exist $v_1(x_n),\ldots,v_n(x_n) \in k[x_n]$ \rm(or $k[[x_n]]$), such that for each $\delta_i \in \mathfrak{D}$ there are scalars $\lambda_i,c_1^i,\ldots,c_{n-1}^i \in k$ such that $$\delta_i=\sum_{l=1}^{n-1}(\lambda_i v_l(x_n)+c_l^i)\partial_{x_l}+ \lambda_i v_n(x_n)\partial_{x_n}.$$

If, in addition, $R$ is $\mathfrak{D}$-simple, then $v_n(x_n)$ must be some nonzero scalar $\beta$ and, also, some $\lambda_j$ is not zero. In this case, $R$ is also simple under the subset $$\left\{\partial_{x_1},\ldots,\partial_{x_{n-1}}, \sum_{l=1}^{n-1}(\lambda_j v_l(x_n)+c_l^j)\partial_{x_l}+ \lambda_j \beta \partial_{x_n}\right\}.$$

\end{thm}

\begin{proof} If either of the conditions is met, is clear that all $k$-derivations of $\mathfrak{D}$ will commute.

Conversely, let $\mathfrak{D}=\{\delta_1,\ldots,\delta_m \}$ be according to the hypotheses of the theorem. By Lemma \ref{lem1.1}, each $\delta_i$ is of the form $\delta_i=p_i(x_n)\partial_{x_1}+q_{(1,i)}(x_n)\partial_{x_2}+\ldots+q_{(n-1,i)}(x_n)\partial_{x_n}.$

If all $q_{(n-1,i)}(x_n)$ are zero, then the first case holds. So, without loss of generality, we assume that $q_{(n-1,1)}(x_n) \neq 0$ and observe that: $$\delta_i(\delta_1(x_n))=\delta_i(q_{(n-1,1)}(x_n))=q_{(n-1,i)}(x_n)\partial_{x_n}(q_{(n-1,1)}(x_n)),$$ $$\delta_1(\delta_i(x_n))=\delta_1(q_{(n-1,i)}(x_n))=q_{(n-1,1)}(x_n)\partial_{x_n}(q_{(n-1,i)}(x_n)).$$

Since $\delta_i$ and $\delta_1$ commute, $$q_{(n-1,i)}(x_n)\partial_{x_n}(q_{(n-1,1)}(x_n))=q_{(n-1,1)}(x_n)\partial_{x_n}(q_{(n-1,i)}(x_n)).$$ Then, this equation must also hold in the ring of fractions, hence we deduce that $$\frac{q_{(n-1,i)}(x_n)\partial_{x_n}(q_{(n-1,1)}(x_n))-q_{(n-1,1)}(x_n)\partial_{x_n}(q_{(n-1,i)}(x_n))}{(q_{(n-1,1)}(x_n))^2}=0.$$
In other words, $$\left(\frac{q_{(n-1,i)}(x_n)}{q_{(n-1,1)}(x_n)}\right)'=0.$$

Then there is some $\lambda_i \in k$ such that $q_{(n-1,i)}(x_n)=\lambda_iq_{(n-1,1)}(x_n).$

Now, we observe that $$\delta_i(\delta_1(x_1))=\delta_i(p_1(x_n))=q_{(n-1,i)}(x_n)\partial_{x_n}(p_1(x_n)),$$ $$\delta_1(\delta_i(x_1))=\delta_1(p_i(x_n))=q_{(n-1,1)}(x_n)\partial_{x_n}(p_i(x_n)).$$

Since $\delta_i$ and $\delta_1$ commute and both $k[x_n]$ and $k[[x_n]]$ are domains, $$\lambda_i\partial_{x_n}(p_1(x_n))=\partial_{x_n}(p_i(x_n)).$$

Then there is some $c_1^i\in k$ such that $p_i(x_n)=\lambda_ip_1(x_n)+c_1^i$. Finally, doing the same argument for the other variables, we prove the desired result.

Now we suppose, in addition, that $R$ is $\mathfrak{D}$-simple. By Lemma \ref{lem1.2}, the greatest common divisor of $q_{(n-1,1)}(x_n),\ldots,q_{(n-1,m)}(x_n)$ must be a unit. However, we have demonstrated that all the $q_{(n-1,i)}(x_n)$ are scalar multiples, then at least one of the $q_{(n-1,i)}(x_n)$ must be a unit, we assume that $q_{(n-1,j)}(x_n)$ is a unit. Since $I$ is stabilized by $$\{\partial_{x_1},\ldots,\partial_{x_{n-1}},p_j(x_n)\partial_{x_1}+q_{(1,j)}(x_n)\partial_{x_2}+\ldots+q_{(n-1,j)}(x_n)\partial_{x_n}\},$$ $I$ must be trivial because, in this case, $I$ is stabilized by $\partial_{x_1},\ldots,\partial_{x_{n-1}}, \partial_{x_{n}}$. Therefore, $R$ is $\{\partial_{x_1},\ldots,\partial_{x_{n-1}},p_j(x_n)\partial_{x_1}+q_{(1,j)}(x_n)\partial_{x_2}+\ldots+q_{(n-1,j)}(x_n)\partial_{x_n}\}$-simple which completes the proof.
\end{proof}

\begin{rem}  A. Nowicki in (\cite[Theorem 2.5.5.]{Now}) proved that every $k$-derivation of a commutative base of $\derk(k[x_1,\ldots,x_n])$ is a special $k$-derivation. This means that the divergence $d^*$ of $d$ is $0$. Moreover, is easy to prove that the set $$\{\partial_{x_1},\ldots,\partial_{x_{n-1}},\sum_{l=1}^{n-1}(\lambda_j v_l(x_n)+c_l^j)\partial_{x_l}+ \lambda_j \beta \partial_{x_n} \}$$ obtained by the previous theorem is a commutative base of $\derk(k[x_1,\ldots,x_n])$. Thus, in particular, $\sum_{l=1}^{n-1}(\lambda_j v_l(x_n)+c_l^j)\partial_{x_l}+ \lambda_j \beta \partial_{x_n}$ is a special derivation. However, this is easily verified since $$d^*=\sum_{l=1}^{n-1} \partial_{x_l}(\lambda_j v_l(x_n)+c_l^j)+\partial_{x_n}(\lambda_j \beta )=0.$$

\end{rem}

\begin{cor}
Let $\mathfrak{D}$ a set of commute $k$-derivations of $R$ such that $R$ is $\mathfrak{D}$-simple and $\partial_{x_1},\ldots,\partial_{x_{n-1}} \in \mathfrak{D}$. Then, there exist $d \in \mathfrak{D}$ and there exist elements $f_1,\ldots,f_n \in R$ such that $d(f_n)=1$ and $d(f_i)=0$, for any $i=1,\ldots,n-1$, so that $R$ is $\{\partial_{x_1},\ldots,\partial_{x_{n-1}},d \}$-simple.
\end{cor}
\begin{proof}
By the previous theorem, we know that there is $d \in \mathfrak{D}$ of the form $$d=\sum_{l=1}^{n-1}(\lambda_j v_l(x_n)+c_l^j)\partial_{x_l}+ \lambda_j \beta \partial_{x_n}$$ such that $R$ is $\{\partial_{x_1},\ldots,\partial_{x_{n-1}},d \}$-simple.
Since $\beta$ and $\lambda_j$, in the theorem, are nonzero scalars, we denote $f_n=(\lambda_j\beta)^{-1}x_n$, thus $d(f_n)=1$.

Let $f_l^* \in k[x_n]$ (or $k[[x_n]]$) such that $d(f_l^*)=\lambda_j v_l(x_n)+c_l^j$. Since $c(k)=0$, $f_l^*$ exist. Then, let $f_l=x_l-f_l^*$, hence $d(f_l)=0$, for any $l=1,\ldots,n-1$. This completes the proof.
\end{proof}

\begin{rem} The previous corollary is a particular case of an important theorem about the characterization of commutative basis of $\derk(R)$ such that $R$ denote either $k[x_1,\ldots,x_n]$ or $k[[x_1,\ldots,x_n]]$, which in our case is more evident its proof (see \cite[Theorem 2.]{Now86}).

\end{rem}
For the remainder of this note we assume that $R$ is the ring $k[x_1,\ldots,x_n]$ of polynomials over $k$ and $\{\partial_{x_1},\ldots,\partial_{x_{n-1}},d \}$ is as in the previous theorem and also we recall the following definitions.

We recall from \cite{Now} that a $k$-derivation $d$ of $k[x_1,\ldots,x_n]$ is called \emph{locally nilpotent} if for each $f \in R$ exists a natural number $n$ such that $d^n(f)=0$ and we say that a basis $\{d_1,\ldots,d_n \}$ of $ \derk(k[x_1,\ldots,x_n])$ is \emph{locally nilpotent} if every derivation $d_i$ is locally nilpotent for $i=1,\ldots,n$.

We remember also that the \emph{Jacobian Conjecture} states that if $F=(F_1,\ldots,F_n)$ is a polynomial map such that the Jacobian matrix is invertible, then $F$ has a polynomial inverse (see \cite{Now}).

\begin{thm}\cite[Theorem.5.]{Now86} Let $R=k[x_1,\ldots,x_n]$ the polynomial ring in $n$ variables over $k$. The following conditions are equivalent:
\begin{enumerate}
\item The Jacobian Conjecture is true in the $n$-variable case.
\item Every commutative basis of the $R$-module $\derk(R)$ is locally nilpotent.
\item Every commutative basis of the $R$-module $\derk(R)$ is locally finite.
\end{enumerate}

\end{thm}

\begin{cor} Let $\mathfrak{D}$ a set of commute $k$-derivations of $R$ such that $R$ is $\mathfrak{D}$-simple, $\partial_{x_1},\ldots,\partial_{x_{n-1}} \in \mathfrak{D}$ and the Jacobian Conjecture is true in $R$. Then, there exist $d \in \mathfrak{D}$
such that $\{\partial_{x_1},\ldots,\partial_{x_{n-1}},d \}$ is a locally nilpotent commutative base of the $R$-module $\derk(R)$. In particular, $d$ is a $k$-derivation locally nilpotent.
\end{cor}
\begin{proof} The proof is immediate consequence of the (\cite[Theorem 5.]{Now86}).
\end{proof}

\begin{ques}A ring is called $w$-\emph{differentially simple} if it is simple relative to a family with $w$ derivations. Recall that we are assuming $R=k[x_1,\ldots,x_n]$, then we know that $R$ is $1$-differentially simple and $\rm dim(R)$-differentially simple as well. However, $n=\rm dim(R)$ is not necessarily the smallest $w$ for which such a ring can be $w$-differentially simple (see \cite{LC}). Thus, one may ask: what is the smallest positive integer $w \neq 1$ such that $R$ is $w$-differentially simple and all $w$ derivations are nonsimple and commute?

\end{ques}

\end{document}